\documentclass[reqno]{amsart}

\usepackage{amssymb}
\usepackage{graphicx}
\usepackage{amscd}
\usepackage[hidelinks]{hyperref}
\usepackage{color}
\usepackage{float}
\usepackage{graphics,amsmath,amssymb}
\usepackage{amsthm}
\usepackage{amsfonts}
\usepackage{latexsym}
\usepackage{epsf}
\usepackage{xifthen}
\usepackage{mathrsfs}
\usepackage{dsfont}
\usepackage{makecell}
\usepackage[FIGTOPCAP]{subfigure}
\usepackage{amsmath}
\allowdisplaybreaks[4]
\usepackage{listings}
\usepackage{etoolbox}
\usepackage{fancyhdr}
\usepackage{pdflscape}
\usepackage[title,toc,titletoc]{appendix}
\usepackage{enumitem}
\usepackage[noadjust]{cite}
\usepackage{tikz}
\usepackage{ytableau}

 \newtheoremstyle{mytheorem}
 {3pt}
 {3pt}
 {\slshape}
 {}
 {\bfseries}
 {.}
 { }
 {}

\numberwithin{equation}{section}

\theoremstyle{theorem}
\newtheorem{theorem}{Theorem}[section]
\newtheorem*{theorem*}{Theorem}
\newtheorem{corollary}[theorem]{Corollary}
\newtheorem{lemma}[theorem]{Lemma}

\theoremstyle{definition}
\newtheorem{definition}{Definition}[section]

\newtheorem*{example*}{Example}

\theoremstyle{remark}

\newtheorem*{remark*}{Remark}
\newtheorem*{remarks*}{Remarks}

\newcommand{\Keywords}[1]{\ifthenelse{\isempty{#1}}{}{\smallskip \smallskip \noindent \textbf{Keywords}. #1}}
\newcommand{\MSC}[2][2010]{\ifthenelse{\isempty{#2}}{}{\smallskip \smallskip \noindent \textbf{#1MSC}. #2}}
\newcommand{\abstractnote}[1]{\ifthenelse{\isempty{#1}}{}{\smallskip \smallskip \noindent \textsuperscript{\dag}#1}}

\makeatletter
\def\specialsection{\@startsection{section}{1}%
  \z@{\linespacing\@plus\linespacing}{.5\linespacing}%
  {\normalfont}}
\def\section{\@startsection{section}{1}%
  \z@{.7\linespacing\@plus\linespacing}{.5\linespacing}%
  {\normalfont\scshape}}
\patchcmd{\@settitle}{\uppercasenonmath\@title}{\Large\boldmath}{}{}
\patchcmd{\@settitle}{\begin{center}}{\begin{flushleft}}{}{}
\patchcmd{\@settitle}{\end{center}}{\end{flushleft}}{}{}
\patchcmd{\@setauthors}{\MakeUppercase}{\normalsize}{}{}
\patchcmd{\@setauthors}{\centering}{\raggedright}{}{}
\patchcmd{\section}{\scshape}{\large\bfseries\boldmath}{}{}
\patchcmd{\subsection}{\bfseries}{\bfseries\boldmath}{}{}
\renewcommand{\@secnumfont}{\bfseries}
\patchcmd{\@startsection}{\@afterindenttrue}{\@afterindentfalse}{}{}
\patchcmd{\abstract}{\leftmargin3pc}{\leftmargin1pc}{}{}

\def\maketitle{\par
  \@topnum\z@ 
  \@setcopyright
  \thispagestyle{empty}
  \ifx\@empty\shortauthors \let\shortauthors\shorttitle
  \else \andify\shortauthors
  \fi
  \@maketitle@hook
  \begingroup
  \@maketitle
  \toks@\@xp{\shortauthors}\@temptokena\@xp{\shorttitle}%
  \toks4{\def\\{ \ignorespaces}}
  \edef\@tempa{%
    \@nx\markboth{\the\toks4
      \@nx\MakeUppercase{\the\toks@}}{\the\@temptokena}}%
  \@tempa
  \endgroup
  \c@footnote\z@
  \@cleartopmattertags
}
\makeatother

\newcommand{\sD}{\mathscr{D}}
\newcommand{\sP}{\mathscr{P}}

\title{On the $k$-measure of partitions and distinct partitions}

\author[G. E. Andrews]{George E. Andrews}
\address[G. E. Andrews]{Department of Mathematics, The Pennsylvania State University, University Park, PA 16802, USA}
\email{gea1@psu.edu}

\author[S. Chern]{Shane Chern}
\address[S. Chern]{Department of Mathematics, The Pennsylvania State University, University Park, PA 16802, USA}
\email{shanechern@psu.edu; chenxiaohang92@gmail.com}

\author[Z. Li]{Zhitai Li}
\address[Z. Li]{Department of Mathematics, The Pennsylvania State University, University Park, PA 16802, USA}
\email{zfl5082@psu.edu}

\begin{document}

\maketitle

{\noindent\footnotesize\itshape Dedicated to Ian Goulden and David Jackson.}

\begin{abstract}

The $k$-measure of an integer partition was recently introduced by Andrews, Bhattacharjee and Dastidar. In this paper, we establish trivariate generating function identities counting both the length and the $k$-measure for partitions and distinct partitions, respectively. The $2$-measure case for partitions extends a result of Andrews, Bhattacharjee and Dastidar.


\MSC{11P84, 05A17.}
\end{abstract}

\section{Introduction}

The study of sequences in partitions has its origin in Sylvester's seminal paper \cite[Item (16), p.~265]{Syl1882}. His result can be translated into the modern language of integer partitions as follows.

\begin{theorem}[Sylvester]
	The number of partitions of $n$ into odd parts with exactly $k$ distinct parts equals the number of partitions of $n$ into distinct parts such that exactly $k$ sequences of consecutive integers occur in each partition.
\end{theorem}

Along this road, we also witness the pioneering work of MacMahon \cite[Section VII, Chapter IV, pp.~49--58]{Max1984}, Hirschhorn \cite[Chapter 5, pp.~51--56]{Hir1979} and many others. Briefly speaking, in these work, the main objects are the sequence or certain subsequences of parts in partitions in which some sort of difference conditions are satisfied. From this perspective, Andrews, Bhattacharjee and Dastidar \cite{ABD2021} recently introduced a family of new statistics for partitions, which they called $k$-measure.

\begin{definition}
	The $k$-measure of a partition is the length of the largest subsequence of parts in a partition in which the difference between any two consecutive parts of the subsequence is at least $k$.
\end{definition}

As remarked in \cite{ABD2021}, the $1$-measure of a partition is the number of distinct parts in this partition. The objective of \cite{ABD2021} is to build a surprising connection between the $2$-measure and the Durfee square of partitions.

\begin{theorem}[Andrews et al.]\label{th:ABD}
	The number of partitions of $n$ with $2$-measure $m$ equals the number of partitions of $n$ with Durfee square of length $m$.
\end{theorem}

Throughout, let $\sP$ and $\sD$ denote the set of partitions and the set of partitions into distinct parts, respectively.

Given any partition $\lambda$, we denote by $|\lambda|$ the size, that is, the sum of all parts of $\lambda$, by $\ell(\lambda)$ the length, that is, the number of parts of $\lambda$, and by $\mu_k(n)$ the $k$-measure of $\lambda$ for positive integer $k$.

Let the $q$-Pochhammer symbol be defined for $n\in \mathbb{N}\cup\{\infty\}$:
$$(A;q)_n:=\prod_{k=0}^{n-1}(1-Aq^k).$$

Our first objective is to establish trivariate generating function identities for partitions that enumerates both the length and the $k$-measure.

\begin{theorem}\label{th:P-main}
	For $k\ge 1$,
	\begin{equation}\label{eq:P-main}
	\sum_{\lambda\in\sP} y^{\ell(\lambda)} z^{\mu_k(\lambda)} q^{|\lambda|} = \frac{1}{(yq;q)_\infty}\sum_{n\ge 0}\frac{(-1)^n y^n q^{n(n+1)/2} (z;q^{k-1})_n}{(q;q)_n}.
	\end{equation}
	Also, assuming that $|z|<1$,
	\begin{equation}\label{eq:P-main-2}
	\sum_{\lambda\in\sP} y^{\ell(\lambda)} z^{\mu_k(\lambda)} q^{|\lambda|} = (z;q^{k-1})_\infty \sum_{n\ge 0}\frac{z^n}{(q^{k-1};q^{k-1})_n (yq;q)_{(k-1)n}}.
	\end{equation}
\end{theorem}

In a combinatorial viewpoint, the $2$-measure case of the above result yields an extension of Theorem \ref{th:ABD}.

\begin{corollary}\label{coro:P}
	The number of partitions of $n$ with $\ell$ parts and $2$-measure equal to $m$ is the same as the number of partitions of $n$ with $\ell$ parts and Durfee square of length $m$.
\end{corollary}

We also have a corollary that builds a connection with partitions into distinct odd parts.

\begin{corollary}\label{coro:P2}
	The excess of the number of partitions $\lambda$ of $n$ with $\ell(\lambda)+\mu_2(\lambda)$ even over those with $\ell(\lambda)+\mu_2(\lambda)$ odd equals the number of partitions of $n$ into distinct odd parts.
\end{corollary}

Our next objective is to investigate an analog for partitions into distinct parts.

\begin{theorem}\label{th:D-main}
	For $k\ge 1$,
	\begin{equation}
	\sum_{\lambda\in\sD} y^{\ell(\lambda)} z^{\mu_k(\lambda)} q^{|\lambda|} = (-yq;q)_\infty \sum_{n\ge 0}\frac{(-1)^n y^n q^n (z;q^k)_n}{(q;q)_n}.
	\end{equation}
	Also, assuming that $|z|<1$,
	\begin{equation}\label{eq:D-main-2}
	\sum_{\lambda\in\sD} y^{\ell(\lambda)} z^{\mu_k(\lambda)} q^{|\lambda|} = (z;q^k)_\infty \sum_{n\ge 0}\frac{(-yq;q)_{kn} z^n}{(q^{k};q^{k})_n}.
	\end{equation}
\end{theorem}

Theorems \ref{th:P-main} and \ref{th:D-main} immediately reveal two unexpected nonnegativity results.

\begin{corollary}\label{coro:nonnegative}
	Let $k\ge 1$. In the series expansions of
	$$\frac{1}{(yq;q)_\infty}\sum_{n\ge 0}\frac{(-1)^n y^n q^{n(n+1)/2} (z;q^k)_n}{(q;q)_n}$$
	and
	$$(-yq;q)_\infty \sum_{n\ge 0}\frac{(-1)^n y^n q^n (z;q^k)_n}{(q;q)_n},$$
	the coefficient of $y^\ell z^m q^n$ is always nonnegative.
\end{corollary}

\section{$q$-Difference equations}

Throughout, we define for $k\ge 1$,
\begin{align*}
F_k(y)=F_k(y,z,q)&:=\sum_{\lambda\in\sP} y^{\ell(\lambda)} z^{\mu_k(\lambda)} q^{|\lambda|},\\
G_k(y)=G_k(y,z,q)&:=\sum_{\lambda\in\sD} y^{\ell(\lambda)} z^{\mu_k(\lambda)} q^{|\lambda|}.
\end{align*}
Our task in this section is to construct $q$-difference equations for $F_k(y)$ and $G_k(y)$.

Let $\sigma(\lambda)$ be the smallest part of a partition $\lambda$. We start with a trivial observation. If we add a positive integer $j$ to each part of $\lambda$, then the resulting partition $\lambda^*$ has $\ell(\lambda^*)=\ell(\lambda)$, $\mu_k(\lambda^*)=\mu_k(\lambda)$ and $\sigma(\lambda^*)\ge j+1$. Therefore,
\begin{align*}
\sum_{\substack{\lambda\in\sP\\\sigma(\lambda)\ge j}} y^{\ell(\lambda)} z^{\mu_k(\lambda)} q^{|\lambda|} &= F_k(yq^{j-1}),\\
\sum_{\substack{\lambda\in\sD\\\sigma(\lambda)\ge j}} y^{\ell(\lambda)} z^{\mu_k(\lambda)} q^{|\lambda|} &= G_k(yq^{j-1}).
\end{align*}
Thus,
\begin{align}
F_k(y)-F_k(yq) &= \sum_{\substack{\lambda\in\sP\\\sigma(\lambda)= 1}} y^{\ell(\lambda)} z^{\mu_k(\lambda)} q^{|\lambda|},\label{eq:F1}\\
G_k(y)-G_k(yq) &= \sum_{\substack{\lambda\in\sD\\\sigma(\lambda)= 1}} y^{\ell(\lambda)} z^{\mu_k(\lambda)} q^{|\lambda|}.\label{eq:G1}
\end{align}

On the other hand, given a partition $\lambda$ with $\sigma(\lambda)=1$, we delete all parts of size $1,2\ldots,k$ and call the resulting partition $\lambda'$. Then, $\sigma(\lambda')\ge k+1$. We also claim that $\mu_k(\lambda)=\mu_k(\lambda')+1$. To see this, we assume that $\mu_k(\lambda')=m$, which means that there exists an $m$-member subsequence $S$ of parts in $\lambda'$ such that the difference between any two consecutive parts of this subsequence is at least $k$. Also, no such $(m+1)$-member subsequence exists. Note that all parts in $S$ are at least $k+1$. Since $1$ is a part in $\lambda$, then $S\cup\{1\}$ gives an $(m+1)$-member subsequence of parts in $\lambda$ so that the requirement of $k$-measure is satisfied. Thus, $\mu_k(\lambda)\ge m+1 = \mu_k(\lambda')+1$. To show $\mu_k(\lambda)\le m+1$, we simply assume that these is an $(m+2)$-member subsequence of parts in $\lambda$ with the requirement of $k$-measure satisfied. Since $1,2,\ldots,k$ cannot be simultaneously in this subsequence, then we arrive at a subsequence of length at least $m+1$ of parts in $\lambda'$ wherein the requirement of $k$-measure is satisfied. This results in a contradiction. Therefore, we are led to the desired claim $\mu_k(\lambda)=\mu_k(\lambda')+1$. Thus,
\begin{align}
\sum_{\substack{\lambda\in\sP\\\sigma(\lambda)= 1}} y^{\ell(\lambda)} z^{\mu_k(\lambda)} q^{|\lambda|}&=\frac{yzq}{(1-yq)(1-yq^2)\cdots (1-yq^k)}\sum_{\substack{\lambda\in\sP\\\sigma(\lambda)\ge k+1}} y^{\ell(\lambda)} z^{\mu_k(\lambda)} q^{|\lambda|}\notag\\
&=\frac{yzq}{(yq;q)_k}F_k(yq^k)\label{eq:F2}
\end{align}
and
\begin{align}
\sum_{\substack{\lambda\in\sD\\\sigma(\lambda)= 1}} y^{\ell(\lambda)} z^{\mu_2(\lambda)} q^{|\lambda|}&=yzq(1+yq^2)(1+yq^3)\cdots (1+yq^k)\sum_{\substack{\lambda\in\sD\\\sigma(\lambda)\ge k+1}} y^{\ell(\lambda)} z^{\mu_2(\lambda)} q^{|\lambda|}\notag\\
&=yzq(-yq^2;q)_{k-1}G_k(yq^k).\label{eq:G2}
\end{align}

In light to \eqref{eq:F1}--\eqref{eq:G2}, we have the following $q$-difference equations.

\begin{lemma}
	For $k\ge 1$,
	\begin{equation}\label{eq:F-dif}
	F_k(y) - F_k(yq) = \frac{yzq}{(yq;q)_k}F_k(yq^k)
	\end{equation}
	and
	\begin{equation}\label{eq:G-dif}
	G_k(y) - G_k(yq) = yzq(-yq^2;q)_{k-1}G_k(yq^k).
	\end{equation}
\end{lemma}

\section{Proofs of Theorem \ref{th:P-main} and Corollaries \ref{coro:P} and \ref{coro:P2}}

We first prove Theorem \ref{th:P-main}.

\begin{proof}[Proof of Theorem \ref{th:P-main}]
	We first multiply by $(yq;q)_\infty$ on both sides of \eqref{eq:F-dif}. Then,
	$$(yq;q)_\infty F_k(y) - (1-yq)(yq^2;q)_\infty F_k(yq) = yzq(yq^{k+1};q)_\infty F_k(yq^k).$$
	Let us define
	$$S_k(y):=(yq;q)_\infty F_k(y).$$
	Thus,
	\begin{align}
	S_k(y)-(1-yq)S_k(yq)=yzqS_k(yq^k).
	\end{align}
	We further write
	$$S_k(y):=\sum_{n\ge 0}s_k(n) y^n.$$
	Then $s_k(0)=S_k(0)=F_k(0)=1$. Further, for $n\ge 1$, we deduce from the above $q$-difference equation for $S_k(y)$ that
	$$s_k(n)-\big(s_k(n)q^n-qs_k(n-1)q^{n-1}\big) = zqs_k(n-1)q^{k(n-1)}.$$
	Therefore,
	$$\frac{s_k(n)}{s_k(n-1)}= \frac{-q^n(1-zq^{(k-1)(n-1)})}{1-q^n}.$$
	We then have
	$$s_k(n)=\frac{(-1)^n q^{n(n+1)/2} (z;q^{k-1})_n}{(q;q)_n},$$
	and therefore,
	$$S_k(y)=\sum_{n\ge 0}\frac{(-1)^n y^n q^{n(n+1)/2} (z;q^{k-1})_n}{(q;q)_n}.$$
	Since $F_k(y)= S_k(y)/(yq;q)_\infty$, the first part of the theorem follows.
	
	For the second part of this theorem, we recall Euler's first and second identities \cite[Corollary 2.2, p.~19]{And1976}:
	\begin{align}
	\sum_{m\ge 0} \frac{t^m}{(q;q)_m}&=\frac{1}{(t;q)_\infty},\label{eq:Eul1}\\
	\sum_{m\ge 0} \frac{(-t)^m q^{m(m-1)/2}}{(q;q)_m}&=(t;q)_\infty.\label{eq:Eul2}
	\end{align}
	Thus,
	\begin{align*}
	F_k(y)&=\frac{1}{(yq;q)_\infty}\sum_{n\ge 0}\frac{(-1)^n y^n q^{n(n+1)/2} (z;q^{k-1})_n}{(q;q)_n}\\
	&=\frac{(z;q^{k-1})_\infty}{(yq;q)_\infty} \sum_{n\ge 0}\frac{(-1)^n y^n q^{n(n+1)/2}}{(q;q)_n}\frac{1}{(zq^{(k-1)n};q^{k-1})_\infty}\\
	&=\frac{(z;q^{k-1})_\infty}{(yq;q)_\infty} \sum_{n\ge 0}\frac{(-1)^n y^n q^{n(n+1)/2}}{(q;q)_n}\sum_{m\ge 0}\frac{z^m q^{(k-1)mn}}{(q^{k-1};q^{k-1})_m}\tag{by \eqref{eq:Eul1}}\\
	&=\frac{(z;q^{k-1})_\infty}{(yq;q)_\infty} \sum_{m\ge 0} \frac{z^m}{(q^{k-1};q^{k-1})_m} \sum_{n\ge 0}\frac{(-1)^n y^n q^{n(n+1)/2+(k-1)mn}}{(q;q)_n}\\
	&=\frac{(z;q^{k-1})_\infty}{(yq;q)_\infty} \sum_{m\ge 0} \frac{z^m}{(q^{k-1};q^{k-1})_m} (yq^{(k-1)m+1};q)_\infty \tag{by \eqref{eq:Eul2}}\\
	&=(z;q^{k-1})_\infty \sum_{m\ge 0} \frac{z^m}{(q^{k-1};q^{k-1})_m (yq;q)_{(k-1)m}}.
	\end{align*}
	This proves \eqref{eq:P-main-2}.
\end{proof}

Corollary \ref{coro:P} is an immediate consequence of Theorem \ref{th:P-main}.

\begin{proof}[Proof of Corollary \ref{coro:P}]
	Letting $k=2$ in \eqref{eq:P-main-2}, we have
	$$\sum_{\lambda\in\sP} y^{\ell(\lambda)} z^{\mu_2(\lambda)} q^{|\lambda|} = (z;q)_\infty \sum_{n\ge 0}\frac{z^n}{(q;q)_n(yq;q)_n}.$$
	Recall Heine's third transformation \cite[(III.3), p.~359]{GR2004}:
	$${}_{2}\phi_1\left(\begin{matrix} a,b\\ c \end{matrix}; q, z\right) = \frac{(abz/c;q)_\infty}{(z;q)_\infty} {}_{2}\phi_1\left(\begin{matrix} c/a,c/b\\ c \end{matrix}; q, \frac{abz}{c}\right),$$
	where the ${}_2 \phi_{1}$ series is defined for $|z|<1$ by
	$${}_{2}\phi_1\left(\begin{matrix} a,b\\ c \end{matrix}; q, z\right):=\sum_{n\ge 0}\frac{(a;q)_n(b;q)_n}{(q;q)_n(c;q)_n} z^n.$$
	Putting $a=yq/\tau$, $b=z/\tau$, $c=yq$ and $z=\tau^2$ and letting $\tau\to 0$ in the above, we find that
	$$\lim_{\tau\to 0}{}_{2}\phi_1\left(\begin{matrix} yq/\tau,z/\tau\\ yq \end{matrix}; q, \tau^2\right)=(z;q)_\infty \sum_{n\ge 0}\frac{z^n}{(q;q)_n(yq;q)_n}.$$
	Thus,
	\begin{align}
	\sum_{\lambda\in\sP} y^{\ell(\lambda)} z^{\mu_2(\lambda)} q^{|\lambda|} &= \lim_{\tau\to 0}{}_{2}\phi_1\left(\begin{matrix} yq/\tau,z/\tau\\ yq \end{matrix}; q, \tau^2\right)\notag\\
	&=\sum_{n\ge 0}\frac{y^n z^n q^{n^2}}{(yq;q)_n (q;q)_n}.
	\end{align}
	
	On the other hand, given a partition $\lambda$, let $D(\lambda)$ be the length of its Durfee square. It suffices to show that
	$$\sum_{\lambda\in\sP} y^{\ell(\lambda)} z^{D(\lambda)} q^{|\lambda|} = \sum_{n\ge 0}\frac{y^n z^n q^{n^2}}{(yq;q)_n (q;q)_n}.$$
	To see this, we simply decompose a partition with Durfee square of length $n$ as in Figure \ref{fig:Durfee}. Then the Durfee square generates $y^nz^nq^{n^2}$; the block below the Durfee square generates $1/(yq;q)_n$; the block to the right of the Durfee square generates $1/(q;q)_n$. The desired identity therefore follows.
	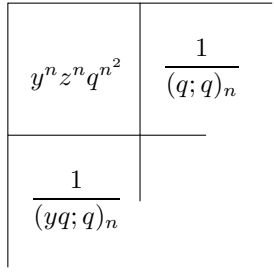
\begin{figure}[ht]
		\caption{Decomposing a partition with Durfee square of length $n$}\label{fig:Durfee}
		\medskip\medskip
		\centering
		\setlength{\unitlength}{1em}
		\ytableausetup{boxsize=1em}
		\begin{tikzpicture}[x=1em,y=1em]
		\draw(0,0)rectangle (5,-5);
		\draw (0,-5) -- (0,-10);
		\draw (5,-5) -- (5,-7.5);
		\draw (5,0) -- (10,0);
		\draw (5,-5) -- (7.5,-5);
		\node[ below right] at (5.5,-1) {$\dfrac{1}{(q;q)_n}$};
		\node[ below right] at (0.5,-6) {$\dfrac{1}{(yq;q)_n}$};
		\node[ below right] at (0.5,-1.75) {$y^nz^nq^{n^2}$};
		\end{tikzpicture}
	\end{figure}
\end{proof}

Lastly, we show Corollary \ref{coro:P2}.

\begin{proof}[Proof of Corollary \ref{coro:P2}]
	Recall the following special case of the Bailey--Daum summation \cite[Corollary 2.7, p.~21]{And1976}:
	\begin{equation}\label{eq:BD}
	\sum_{n\ge 0}\frac{(a;q)_n q^{n(n+1)/2}}{(q;q)_n} = (-q;q)_\infty (aq;q^2)_\infty.
	\end{equation}
	Setting $k=2$ and $y=z=-1$ in \eqref{eq:P-main}, then
	\begin{align*}
	\sum_{\lambda\in\sP} (-1)^{\ell(\lambda)+\mu_2(\lambda)} q^{|\lambda|}&=\frac{1}{(-q;q)_\infty}\sum_{n\ge 0}\frac{(-1;q)_n q^{n(n+1)/2}}{(q;q)_n}\\
	&= \frac{(-q;q)_\infty (-q;q^2)_\infty}{(-q;q)_\infty} \tag{by \eqref{eq:BD}}\\
	&=(-q;q^2)_\infty.
	\end{align*}
	This immediately implies the desired result.
\end{proof}

\section{Proof of Theorem \ref{th:D-main}}

In this section, we proceed with a proof of Theorem \ref{th:D-main}.

\begin{proof}[Proof of Theorem \ref{th:D-main}]
	We first divide by $(-yq^2;q)_\infty$ on both sides of \eqref{eq:G-dif}. Then,
	$$(1+yq)\frac{G_k(y)}{(-yq;q)_\infty} - \frac{G_k(yq)}{(-yq^2;q)_\infty} = yzq\frac{G_k(yq^k)}{(-yq^{k+1};q)_\infty}.$$
	Let us define
	$$T_k(y):=\frac{G_k(y)}{(-yq;q)_\infty}.$$
	Thus,
	\begin{align}
	(1+yq)T_k(y)-T_k(yq)=yzqT_k(yq^k).
	\end{align}
	Let us further define
	$$T_k(y):=\sum_{n\ge 0}t_k(n) y^n.$$
	Then $t_k(0)=T_k(0)=G_k(0)=1$. Further, for $n\ge 1$, the above $q$-difference equation for $T_k(y)$ yields
	$$\big(t_k(n)+qt_k(n-1)\big)-t_k(n)q^n = zqt_k(n-1)q^{k(n-1)},$$
	that is,
	$$\frac{t_k(n)}{t_k(n-1)}= \frac{-q(1-zq^{k(n-1)})}{1-q^n}.$$
	We then arrive at
	$$t_k(n)=\frac{(-1)^n q^n (z;q^k)_n}{(q;q)_n},$$
	and therefore,
	$$T_k(y)=\sum_{n\ge 0}\frac{(-1)^n y^n q^n (z;q^k)_n}{(q;q)_n}.$$
	Recalling that $G_k(y)=(-yq;q)_\infty T_k(y)$, we are led to the first part of the theorem.
	
	For the second part, we carry out a similar calculation as for \eqref{eq:P-main-2}. We have
	\begin{align*}
	G_k(y)&=(-yq;q)_\infty\sum_{n\ge 0}\frac{(-1)^n y^n q^{n} (z;q^{k})_n}{(q;q)_n}\\
	&=(-yq;q)_\infty (z;q^{k})_\infty \sum_{n\ge 0}\frac{(-1)^n y^n q^{n}}{(q;q)_n}\frac{1}{(zq^{kn};q^{k})_\infty}\\
	&=(-yq;q)_\infty (z;q^{k})_\infty \sum_{n\ge 0}\frac{(-1)^n y^n q^{n}}{(q;q)_n}\sum_{m\ge 0}\frac{z^m q^{kmn}}{(q^{k};q^{k})_m}\tag{by \eqref{eq:Eul1}}\\
	&=(-yq;q)_\infty (z;q^{k})_\infty \sum_{m\ge 0} \frac{z^m}{(q^{k};q^{k})_m} \sum_{n\ge 0}\frac{(-1)^n y^n q^{n+kmn}}{(q;q)_n}\\
	&=(-yq;q)_\infty (z;q^{k})_\infty \sum_{m\ge 0} \frac{z^m}{(q^{k};q^{k})_m} \frac{1}{(-yq^{km+1};q)_\infty} \tag{by \eqref{eq:Eul1}}\\
	&=(z;q^{k})_\infty \sum_{m\ge 0} \frac{(-yq;q)_{km} z^m}{(q^{k};q^{k})_m}.
	\end{align*}
	This proves \eqref{eq:D-main-2}.
\end{proof}

\section{Conclusion}

The succinctness of the expansions appearing in Theorems \ref{th:P-main} and \ref{th:D-main} suggests that the underlying combinatorics merits further investigation. Indeed, Corollaries \ref{coro:P}, \ref{coro:P2} and \ref{coro:nonnegative} call for combinatorial proofs. Also of significant interest is the fact that all the expansions in Theorems \ref{th:P-main} and \ref{th:D-main} are instances of the generalized Heine transformation:
\begin{equation}
\sum_{n\ge 0}\frac{(a;q^h)_n(b;q)_{hn} t^n}{(q^h;q^h)_n(c;q)_{hn}} = \frac{(b;q)_\infty (at;q^h)_\infty}{(c;q)_\infty (t;q^h)_\infty} \sum_{n\ge 0} \frac{(c/b;q)_n(t;q^h)_{n} b^n}{(q;q)_n(at;q^h)_{n}},
\end{equation}
which first appeared in \cite[Lemma 1, p.~577]{And1966} and plays a substantial role in the exploration of many of Ramanujan's identities from his Lost Notebook \cite[p.~6]{AB2009}. This suggests further combinatorial studies of the generalized Heine transformation.

\bibliographystyle{amsplain}

\end{document}